\documentclass{svjour3}                     
\smartqed  
\usepackage{graphicx}
\usepackage{mathptmx}      
\usepackage{amsmath}
\usepackage{amssymb}
\usepackage[usenames]{xcolor}   

\newcommand{\norm}[1]{\|#1\|}
\newcommand{\eoproof}{\hfill$\square$}
\newtheorem{algorithm}{\normalfont\textsc{Algorithm}}

\begin{document}

\title{Sequential Convex Programming Methods for Solving Nonlinear Optimization Problems with DC constraints}



\author{Tran Dinh Quoc \and Moritz Diehl}


\institute{ Tran Dinh Quoc \and Moritz Diehl \at
            Department of Electrical Engineering (ESAT-SCD) and Optimization in Engineering Center (OPTEC), K.U. Leuven, Kasteelpark Arenberg 10, B-3001 Leuven, Belgium.\\
            \email{\{quoc.trandinh, moritz.diehl\}@esat.kuleuven.be}
}

\date{Received: date / Accepted: date}

\maketitle

\begin{abstract}
This paper investigates the relation between sequential convex programming (SCP) as, e.g., defined in \cite{Quoc2009b} and DC (difference of two convex functions) programming. 
We first present an SCP algorithm for solving nonlinear optimization problems with DC constraints and prove its convergence. Then we combine the proposed algorithm with a relaxation technique to handle inconsistent linearizations. Numerical tests are performed to investigate the behaviour of the class of algorithms.
\end{abstract}

\section{Introduction}\label{sec:intro}
Let $\Gamma_0(\mathbf{R}^n)$ denote the set of all proper lower semi-continuous convex functions from $\mathbf{R}^n$ to $\mathbf{R}$, and $\mathcal{DC}(\mathbf{R}^n) := \Gamma_0(\mathbf{R}^n)-\Gamma_0(\mathbf{R}^n)$ denote the set of DC functions on $\mathbf{R}^n$.
We are interested in the following nonconvex optimization problem:
\begin{equation}\label{eq:nlp_prob}
\left\{\begin{array}{cl}
\displaystyle\min_{x\in\mathbf{R}^n} &f(x) \\
\textrm{s.t.} &g(x) \leq 0,\\
              &x\in\Omega,
\end{array}\right.\tag{P}
\end{equation}
where $f : \mathbf{R}^n  \to \mathbf{R}$ is convex, $\Omega$ is a nonempty closed convex subset in $\mathbf{R}^n$, and  $g: \mathbf{R}^n \to \mathbf{R}^m$ with $g=(g_1,\dots, g_m)^T$ and $g_i$ ($i=1,\dots, m$) belongs to $\mathcal{DC}(\mathbf{R}^n)$. We refer to $g(x)\leq 0$ as DC constraints.
Let us denote by $D:=\{x\in\Omega ~:~ g(x)\leq 0\}$ the feasible set of \eqref{eq:nlp_prob} and $\textrm{int}D$ the set of interior points of $D$.

Problems of the form \eqref{eq:nlp_prob} have been studied by many researchers in theory and applications (see, e.g., \cite{An1999,Hiriart-Urruty1986,Hiriart-Urruty1993,Horst1999,Horst2000,Tuy1997} and the references quoted therein). However, the methods for solving \eqref{eq:nlp_prob} that exploit DC structures are usually global optimization techniques. These approaches are not applicable to problems with a high dimension. 
In this paper, we are interested in finding local minimizers only.

The class of DC functions is sufficiently rich to deal with many practical problems.
It is well-known \cite{Horst1999,Horst2000} that the set of DC functions defined on a compact convex set of $\mathbf{R}^n$ is dense in the set of continuous functions on this set. Therefore, in principle, every continuous function can be approximated by a DC function with any desired precision. Moreover, every $C^2$-function defined on a compact set is a DC function \cite{Hiriart-Urruty1986} that includes the smooth cases of \eqref{eq:nlp_prob}. Many practical problems can be reformulated in the form of \eqref{eq:nlp_prob} (see, e.g., \cite{Horst2000}). Although DC representations are available for important function classes, finding such a representation for an arbitrary DC function is still a hard problem.

This paper investigates the relation between SCP methods \cite{Lewis2008,Quoc2009b} and DC programming \cite{An2005,An1999,Pham1998}. Both families of methods address the local solution of nonconvex optimization problems via an iteration based on convex subproblems.

\subsection{DC programming}
DC programming algorithms (DCA) for solving \eqref{eq:nlp_prob} have been introduced by Pham \cite{An2005,An1999,Pham1998}. The original DCA is supposed to solve convex constrained DC programs. To handle DC constraints, penalty functions have been used  \cite{An1999} and then DCA is applied to the penalized problem for a fixed penalty parameter.
Yuillie and Rangarajan in \cite{Yuille2003} proposed a method for solving smooth DC programs that is called the concave-convex procedure (CCCP), a variant of DCA applied to smooth DC programs \cite{Sripertumbudur2007}.
The authors in \cite{Sriperumbudur2009} further investigated the global convergence of the CCCP method.
DCA as well as CCCP have been widely applied in many practical problems (see, e.g. \cite{An2005,Sripertumbudur2007,Yuille2003}).
It is well-known that the use of penalty functions in DC programming with DC constraints introduces conservatism and might lead to excessively short steps.

One particular variant of DC programming that keeps the DC constraints in the problem was considered in \cite{Smola2005}. This again leads to possible conservatism or even to infeasibility of the subproblems (which might be overcome by relaxation techniques). These methods have not become very popular due to these problems and their combination with exact penalties was never properly investigated.

It is the aim of this paper to improve and investigate the numerical behaviour of these algorithms and show that they can be interpreted as a special case of SCP methods. 

\subsection{Sequential Convex Programming}
In \cite{Quoc2009b}, a generic algorithm framework for solving nonlinear optimization problems with partially convex structure was proposed which is called \textit{sequential convex programming} (SCP). 
The main idea of SCP methods is to convexify the nonconvex part and preserve the remaining convexity in the resulting subproblems at each iteration. 
Under mild assumptions, the local convergence of the SCP methods was proved. The rate of local convergence is linear.

To the family of SCP methods belong such classical algorithms as the constrained or unconstrained Gauss-Newton methods as well as sequential linear programming (SLP) or sequential quadratic programming (SQP) with convex subproblems \cite{Garces2008,Kanzow2005,Nocedal2006}. All these methods are based on linearization of nonconvex constraints or objective functions, and are widely used in applications of nonlinear optimization, in particular, in parameter estimation (constrained Gauss-Newton \cite{Bock1986} and nonlinear model predictive control \cite{Diehl2002b,Diehl2009c}).

When DC constraints are treated within an SCP framework, it is possible to only linearize the concave parts. This can be interpreted as a special case of SCP, which offers a favourable feature: namely that globalization strategies like line search or trust-region methods are not needed and full SCP steps can always be taken. When feasibility of the subproblems becomes an issue, which is always the case for nonlinear equality constraints, we propose to relax the subproblems using an exact $L_1$-penalty function and investigate the behaviour of this relaxed SCP algorithm. 
We show through an example that it can lead to less conservative convex subproblems than the standard approach of using unconstrained DC programming with penalty functions.

\subsection{Notation and definitions}
Throughout this paper, we use $\mathbf{R}^m_{+}$ for the set of $m$-nonnegative vectors and $\mathbf{R}_{+}$ (resp., $\mathbf{R}_{+}$) for the set of nonnegative (resp., positive) numbers. 

A function $f:\mathbf{R}^n\to\mathbf{R}$ is called $\rho^f$-convex on a convex subset $X$ of $\mathbf{R}^n$ with $\rho^f\in\mathbf{R}_{+}$ if for all $x,y\in X$ and $t\in [0,1]$  the inequality $f(tx+(1-t)y) \leq tf(x) + (1-t)f(y) - \frac{\rho^f}{2}t(1-t)\norm{x-y}^2$ holds. 
If $\rho^f=0$ then $f$ is convex. Otherwise, $f$ is strongly convex with the parameter $\rho^f > 0$.
 
Let us assume that $f$ is a DC function such that $f = f_1-f_2$, then it is trivial to see that $f = (f_1 + \frac{\rho}{2}\norm{\cdot}^2) - (f_2 + \frac{\rho}{2}\norm{\cdot}^2)$ for any given $\rho>0$. Therefore, without loss of generality, we can find a DC decomposition $(f_1, f_2)$ of $f$ such that $f_1$ and $f_2$ are strongly convex. 
We also use the notation $\textrm{dom}f := \{x\in X~|~ f(x) <+\infty \}$ for the domain of a convex function $f$. For $x\in\textrm{dom}f$, the symbol $\partial f(x)$ denotes the exact subdifferential of $f$ at $x$, i.e., $\partial f(x) := \{ \xi\in\mathbf{R}^n ~|~ f(y) \geq f(x) + \xi^T(y-x) , ~\forall y\in X \}$. A convex function $f$ is said to be subdifferentiable at $x\in\textrm{dom}f$ if $\partial f(x) \neq\emptyset$. A vector $\xi\in\partial f(x)$ is called a subgradient of $f$ at $x$.  

\subsection{Optimality condition}
Suppose that $(u, v)$ is an arbitrary DC decomposition of $g$.    
Let us define $L(x,\lambda) := \lambda_0f(x) + \lambda^T[u(x)-v(x)]$ the Lagrange function of problem \eqref{eq:nlp_prob}. The generalized F. John condition of \eqref{eq:nlp_prob} is expressed as follows \cite{Clarke1990}:
\begin{equation}\label{eq:kkt_cond}
\begin{cases}
0 \in \lambda_0\partial f(x) + \sum_{i=1}^m\lambda_i[\partial u_i(x)- \partial v_i(x)] + N_{\Omega}(x),\\
0 \neq (\lambda_0, \lambda) \geq 0, ~u(x)-v(x) \leq 0, ~\lambda^T[u(x) - v(x)] = 0,
\end{cases} 
\end{equation}
where $\partial f(x)$, $\partial u_i(x)$ and $\partial v_i(x)$ ($i=1,\dots, m$) are the subdifferentials of $f$, $u_i$ and $v_i$ at $x$, respectively. The multivalued mapping $N_{\Omega}$ is the normal cone of $\Omega$ at $x$ defined by:
\begin{equation}\label{eq:normal_cone}
N_{\Omega}(x) := \begin{cases}
\{ w\in\mathbf{R}^n ~|~ w^T(y-x)\leq 0,~ y\in\Omega\} ~&\textrm{if}~ x\in\Omega,\\
\emptyset ~&\textrm{otherwise}. 
\end{cases} 
\end{equation}
Note that the first line of \eqref{eq:kkt_cond} includes implicitly that $x\in\Omega$.
If $(x^{*},\lambda^{*}_0, \lambda^{*})$ satisfies \eqref{eq:kkt_cond} then $x^{*}$ is called a stationary point and $(\lambda_0^{*}, \lambda^{*})$ is the corresponding multiplier of \eqref{eq:nlp_prob}. 

Since problem \eqref{eq:nlp_prob} is nonconvex, a stationary point might not be a local minimizer. However, we will show later that under the calmness constraint qualification, the first order necessary condition for \eqref{eq:nlp_prob} still holds.
   
We consider the following parametric optimization problem: 
\begin{equation}\label{eq:para_nlp}
V(\delta) := \inf\big\{ f(x)~|~g(x) \leq \delta, ~x\in\Omega\big\} \tag{$\textrm{P}(\delta)$}, 
\end{equation}
where the perturbation (or parameter) $\delta$ belongs to a neighborhood $U_{\varepsilon}\subset\mathbf{R}^m$ of the origin. It is trivial that $\textrm{P}(0)\equiv\textrm{P}$. Let $x^{*}$ solve \eqref{eq:nlp_prob}. Problem \eqref{eq:nlp_prob} is said to be calm at $x^{*}$ (in the sense of Clarke's calmness constraint qualification \cite{Clarke1990}) if there exist a neighborhood $U_{\varepsilon}$ of the origin, $X_{\varepsilon}$ of $x^{*}$ and a positive number $\tau$ such that for all $\delta\in U_{\varepsilon}$ and $x\in X_{\varepsilon}$ that are feasible for \ref{eq:para_nlp}, one has $f(x) - f(x^{*}) + \tau\norm{\delta}\geq 0$. The characterizations of calmness have been investigated in the literature (see, e.g., \cite{Clarke1990,Klatte2001,Rockafellar1997}). The optimality conditions for DC programs with DC constraints have been studied in \cite{Laghdir2005}. 

If $v_i$ $(i=1,\dots, m)$ is continuously differentiable on $\mathbf{R}^n$ then, under the calmness of \eqref{eq:nlp_prob} at a local solution $x^{*}$, without loss of generality, we can assume that the multiplier $\lambda_0 = 1$. Thus the F. John condition \eqref{eq:kkt_cond} collapses to the (generalized) KKT condition. With $\lambda_0=1$, the point $(x^{*}, \lambda^{*})$ satisfying \eqref{eq:kkt_cond} is called a KKT point.
In particular, if $f$, $u_i$ and $v_i$ $(i=1,\dots, m)$ are continuously differentiable on $\mathbf{R}^n$, and $\Omega$ is the whole space, then the condition \eqref{eq:kkt_cond} collapses to the classical KKT condition in smooth nonlinear optimization \cite{Nocedal2006}. Under the Mangasarian-Fromowitz constraint qualification, the first order necessary condition corresponding to \eqref{eq:kkt_cond} holds for \eqref{eq:nlp_prob}. 
The following theorem shows that the first order necessary condition for problem \eqref{eq:nlp_prob} still holds.

\begin{theorem}\label{th:FONC}
Suppose that $f\in\Gamma(\mathbf{R}^n)$ and $(u, v)$ is a DC decomposition of $g$ such that $v$ is continuously differentiable on $\mathbf{R}^n$. Let $x^{*}$ be a local minimum of \eqref{eq:nlp_prob} such that \eqref{eq:nlp_prob} is calm at $x^{*}$. Then there exists a multiplier $\lambda^{*}\in\mathbf{R}^m$ such that $(x^{*},\lambda^{*})$ is a solution to the KKT system \eqref{eq:kkt_cond}.
\end{theorem}

\begin{proof}
Note that if a function $\varphi$ is continuously differentiable (resp., convex) then the Clarke subdifferential coincides with its gradient (resp., its convex subdifferential) \cite{Clarke1990}[Proposition 2.2.7]. 
Since $v_i(\cdot)$ is convex and continuously differentiable on $\mathbf{R}^n$, it implies that $\partial v_i = \{\nabla v_i\}$ for all $i=1,\dots, m$.
On the other hand, since $u_i$ is subdifferentiable on $\mathbf{R}^n$, we have $\partial^c g_i = \partial^c( u_i-v_i ) = \partial u_i + \nabla(-v_i ) = \partial u_i - \nabla v_i$, where $\partial^c g_i$ is the Clarke subdifferential of $g_i$ ($i=1,\dots,m$) \cite{Clarke1990}. Applying Proposition 6.4.4 in \cite{Clarke1990} we obtain the conclusion of the theorem.
\eoproof 
\end{proof}

The rest of the paper is organized as follows. Section~\ref{sec:examples} presents two motivating examples. A variant of the SCP algorithm for solving \eqref{eq:nlp_prob} is presented in Section~\ref{sec:SCP_using_inner_approx}. Then its global convergence is investigated in Section~\ref{sec:convergence}.
A relaxation technique is proposed in Section~\ref{sec:SCP_penalty} to handle  possibly inconsistent linearizations. Computational tests are performed in the last section to demonstrate the behaviour of the class of algorithms. 

\section{Motivating examples}\label{sec:examples}
There are many practical problems that can be conveniently reformulated in the form of \eqref{eq:nlp_prob} such as mathematical programs with complementarity constraints, bridge location problems, design centering problems, location problems, packing problems, optimization over efficient sets, trust-region subproblems in SQP algorithms, and nonconvex quadratically constrained quadratic programming problems (see, e.g., \cite{Horst1999,Pham1998}).
For motivation, we present here two examples. The first example originates from optimal control of a bilinear system and the second one is a mathematical programming problem with complementarity constraints.

\subsection{Nonlinear model predictive control (NMPC) of a bilinear system}\label{ex:motivating_exam1}
The optimization problem resulting from NMPC of a bilinear dynamic system has the following form:  
\begin{equation}\label{eq:example_01}
\left\{\begin{array}{cl}
\displaystyle \min_{x, u} & F_0(x, u):= \frac{1}{2}\sum_{k=0}^{H_p-1}[x_k^TW^k_xx_k + u_k^TW^k_uu_k] + \frac{1}{2}x_{H_p}^TW_ex_{H_p} \\
\textrm{s.t.} & x_{k+1} = Ax_k + B[x_k, u_k] + Cu_k, ~ k=0,\dots, H_p-1, \tag{NMPC}\\
& x_0 = x_{\textrm{init}},\notag\\
& \underline{x}_k \leq x_k \leq \bar{x}_k, ~k=0,\dots, H_p, \notag\\ 
&\underline{u}_k \leq u_k \leq \bar{u}_k, ~k=0,\dots, H_p-1, \notag\\
& x_{H_p}^TW_ex_{H_p} \leq r_f,
\end{array}\right.
\end{equation}
where $W_x^k$, $W_u^k$, $W_e$ are the weighting matrices; $A, C$ are given consistent matrices; $x_{\textrm{init}}$ is a given initial state; $\underline{x}_k$, $\bar{x}_k$, $\underline{u}_k$, $\bar{u}_k$ are lower and upper bounds on the variables $x_k$ and $u_k$, respectively; $r_f>0$ is the radius of the terminal region; and $B[x_k, u_k]$ denotes a bilinear form of $x_k$ and $u_k$.  
    
Introducing a new variable $w := (x_0^T, x_1^T, \dots, x_{H_p}^T, u_0^T, \dots, u_{H_p-1}^T)^T \in \mathbf{R}^{n_w}$ with $n_w=(H_p+1)n_x+H_pn_u$, the objective function of \eqref{eq:example_01} can be rewritten as $F_0(w) = \frac{1}{2}w^THw$, where $H$ is a symmetric positive semidefinite matrix with $W^k_x$, $W^k_u$ and $W^k_e$ on the diagonal block. 
It is known that a given bilinear form is always associated with a quadratic form.  
Therefore, the discrete time bilinear dynamic system $x_{k+1} = Ax_k + B[x_k, u_k] + Cu_k$ ~($k=0,\dots, H_p-1$) can be reformulated as:
\begin{equation}\label{eq:dynamic_system}
w^TP_iw + q_i^Tw + r_i = 0, ~~ i=1,\dots, m, 
\end{equation}
where $m := H_pn_x$, $P_i$ is a given symmetric indefinite matrix, $q_i\in\mathbf{R}^{n_w}$ and $r_i \in \mathbf{R}$ ($i=1,\dots, m$). 
Any symmetric indefinite matrix $P_i$ can be decomposed in such a form $P_i := P^1_i - P^2_i$, where $P^1_i$ and $P^2_i$ are two symmetric positive semidefinite matrices (e.g., using spectral decomposition).
Using two different DC decompositions of $P_i$ and choosing $q_i^1$, $q^2_i$, $\tilde q_i^1$, $\tilde q^2_i$, $r_i^1$, $r^2_i$, $\tilde r_i^1$, $\tilde r^2_i$ such that $q_i = q_i^1-q^2_i= \tilde q_i^1-\tilde q^2_i$, $r_i = r_i^1-r^2_i= \tilde r_i^1-\tilde r^2_i$, respectively, the equality constraints \eqref{eq:dynamic_system} can be rewritten as
\begin{equation}\label{eq:dc_decp_for_BLSys}
\begin{cases}
[(w^TP_i^1w + (q^1_i)^Tw + r^1_i] - [w^TP_i^2w + (q_i^2)^Tw + r_i^2] \leq 0,\\
[(w^T\tilde P_i^2w + (\tilde q^2_i)^Tw + \tilde r^2_i] - [w^T\tilde P_i^1w + (\tilde q_i^1)^Tw + \tilde r_i^1] \leq 0, 
\end{cases}
\end{equation}
for all $i=1,\dots, m$.
Hence, problem \eqref{eq:example_01} is reformulated in the form of \eqref{eq:nlp_prob}.
Note that $P^j_i=\tilde{P}_i^j$ ($j=1,2$) is a possible choice in the formula \eqref{eq:dc_decp_for_BLSys}.

\subsection{Mathematical programs with complementarity constraints}\label{ex:motivating_exam2}
Mathematical programs with equilibrium constraints (MPEC) have been studied widely and have many applications in economic models, shape optimization, transportation, network design, and data mining. In this example, we particularly consider the following mathematical programming problem with complementary constraints:
\begin{equation}\label{eq:example_02}
\left\{\begin{array}{cl}
\displaystyle \min_{x\in\mathbf{R}^n, y\in\mathbf{R}^m} & f(x,y)\\ 
\textrm{s.t.} & (x, y) \in S, \tag{MPCC}\\
& x\geq 0,~ Cx + Dy + e \geq 0,\\
& x^T(Cx + Dy + e) = 0,
\end{array}\right.
\end{equation}
where $f:\mathbf{R}^n\times \mathbf{R}^m \to\mathbf{R}$ is convex, $S\subseteq\mathbf{R}^{n+m}$ is a nonempty closed convex set, $e\in\mathbf{R}^n$,  and $C$, $D$ are two given matrices of consistent dimensions. 

Theory and methods for \eqref{eq:example_02} have been developed intensively in recent years (see, e.g., \cite{Facchinei1999,Luo1996,Outrata1999} and the references quoted therein). 
The main difficulty of this problem is the complementarity constraints in the two last lines of \eqref{eq:example_02}. These constraints lead to nonconvexity and loss of constraint qualification of the problem. 

Introducing a slack variable $z$, the complementarity constraint can be reformulated as:
\begin{equation}\label{eq:example_02cc}
Cx + Dy - z + e = 0, ~ x\ge 0,~ z\geq 0, ~ x^Tz = 0.
\end{equation}
Since $x\geq 0$ and $z\geq 0$, the constraint $x^Tz = 0$ is equivalent to $x^Tz\leq 0$. Using the expression $2x^Tz = \norm{x}^2 + \norm{z}^2 - \norm{x-z}^2$, we can rewrite the condition $x^Tz\leq 0$ as a DC constraint:
\begin{equation}\label{eq:dc_con2}
u(x, z) - v(x, z) \leq 0, 
\end{equation}
where $u(x, z) := \norm{(x, z)}^2$ and $v(x, z) := \norm{x-z}^2$ that are convex.
Problem \eqref{eq:example_02} is now reformulated  in the form of \eqref{eq:nlp_prob}.
 
For an MPEC problem, by using the KKT condition for the equilibrium constraint (low level problem), we can transform this problem to the form (MPCC) (see \cite{Facchinei1999}). Then, by the same technique as before, we obtain a DC formulation for the equilibrium constraint.

\section {Sequential convex programming algorithm with DC constraints}\label{sec:SCP_using_inner_approx}
In this section, we present an algorithm for solving problem \eqref{eq:nlp_prob} which we might call \textit{sequential convex programming with DC constraints}. 
Let us assume that $(u, v)$ is a DC decomposition of $g$, i.e.,
\begin{equation}\label{eq:dc_decomp} 
g(x) = u(x) - v(x). 
\end{equation}
For a given point $x^k\in\Omega$, we take an arbitrary matrix $\Xi^k \in \partial v(x^k)$, where the multivalued mapping $\partial v(x^k) := (\partial v_1(x^k)^T, \dots, \partial v_m(x^k)^T)^T$ with $\partial v_i(x^k)$ ($i=1,\dots,m$) is the subdifferential of the convex function $v_i$ at $x^k$. We will refer to $\Xi^k$ as a subgradient matrix of $v$ at $x^k$.
Consider the following convex problem:
\begin{equation}\label{eq:convex_subprob}
\left\{\begin{array}{cl}
\displaystyle\min_{x\in\mathbf{R}^n} &f(x) \\
\textrm{s.t.} &u(x) - v(x^k) - \Xi^k(x-x^k) \leq 0, \tag{\textrm{P}($x^k$)}\\
& x\in\Omega.
\end{array}\right.
\end{equation}
Since \ref{eq:convex_subprob} is convex, under the Slater constraint qualification
\begin{equation}
\text{ri}(\Omega)\cap\left\{x : u(x) - v(x^k) - \Xi^k(x-x^k) < 0 \right\} \neq\emptyset,
\end{equation}
where $\text{ri}(\Omega)$ is the set of relative interior points of the convex set $\Omega$, any global solution $x^{k+1}$ of \ref{eq:convex_subprob} is characterized as a KKT point of \ref{eq:convex_subprob}. In the following algorithm, we assume that the convex subproblem \ref{eq:convex_subprob} is solvable for given $x^k$ and $\Xi^k$.

A generic framework of the \textit{sequential convex programming algorithm with DC constraints} (SCP-DC) can be described as follows:

\noindent\rule[1pt]{\textwidth}{1.0pt}{~~}
\begin{algorithm}\vskip -0.3cm\label{alg:A1}{~}\end{algorithm}
\vskip -0.4cm
\noindent\rule[1pt]{\textwidth}{0.5pt}
\noindent{\bf Initialization:} Take an initial point $x^0$ in $\Omega$. Set $k:=0$.\\
\noindent{\bf Iteration $k$:} For a given $x^k$, execute the three steps below:
\begin{itemize}
\item[]\textit{Step 1}: Compute a subgradient matrix $\Xi^k \in \partial v(x^k)$.
\item[]\textit{Step 2}: Solve the convex subproblem \ref{eq:convex_subprob} to get a solution $x^{k+1}$ and the corresponding multiplier $\lambda^{k+1}$.
\item[]\textit{Step 3}: If $\norm{x^{k+1} - x^k} \leq \varepsilon$ with a given tolerance $\varepsilon>0$, then stop. Otherwise, increase $k$ by $1$ and go back to Step 1.
\end{itemize}
\vskip-0.3cm
\noindent\rule[1pt]{\textwidth}{1.0pt}
At Step 1 of Algorithm \ref{alg:A1}, a subgradient matrix $\Xi^k$ of $v$ at $x^k$ must be computed. If $v_i$ ($i=1,\dots,m$) has a simple form, $\Xi^k$ can be computed explicitly. Otherwise, a convex problem needs to be solved. If $v$ is differentiable at $x^k$ then $\partial v(x^k)$ is identical to the Jacobian matrix of $v$ at $x^k$, i.e. $\partial v(x^k) = \{\nabla v(x^k)\}$.

The cost of finding an initial point $x^0\in\Omega$ depends on the structure of $\Omega$. It can be computed explicitly if $\Omega$ is simple. Otherwise, a convex problem should be solved. The projection methods (onto $\Omega$) can be also used in this case. 

\begin{remark}\label{re:quadratic case}
If the objective function $f$ of \eqref{eq:nlp_prob} is linear (resp., quadratic) then:
\begin{itemize}
\item[i)] If the function $u$ is linear then subproblem \ref{eq:convex_subprob} is linear (resp., quadratic).
\item[ii)]If the function $u$ is quadratic then \ref{eq:convex_subprob} is a quadratically constrained linear (resp., quadratic) programming problem. This problem can be reformulated as a second order cone programming or semidefinite programming problem \cite{Boyd2004}. 
\end{itemize}
\end{remark}

DC decomposition of the function $g$ plays a crucial role in Algorithm \ref{alg:A1}. A suitable DC decomposition may ensure that the convex subproblem \ref{eq:convex_subprob} is solvable. Moreover, it might make \ref{eq:convex_subprob} easy to solve, and help Algorithm \ref{alg:A1} to reach a KKT point of \eqref{eq:nlp_prob} (e.g., $u$ and $v$ have small strongly convex parameters).
The following small example shows the behaviour of Algorithm \ref{alg:A1} using two different DC decompositions.
\begin{equation}\label{eq:small_exam}
\left\{\begin{array}{cl}
\displaystyle\min_{x\in\mathbf{R}^2}& f(x) := -4x_1 + x_2 \\
\textrm{s.t.} &g(x) := x_1^2-x_2^2 - 4\leq 0,\\
& x \in \Omega := [-3, 3]\times [-2, 2]. 
\end{array}\right.
\end{equation}
The constraint $x_1^2-x_2^2-4 \leq 0$ is a DC constraint. Hence, for a given tolerance $\varepsilon = 10^{-5}$ and a starting point $x^0=(0,0)^T$, if we choose $(u,v)$ with $u(x) := x_1^2$ and $v(x) := x_2^2$ for the DC decomposition of $g$ (Case 1) then Algorithm \ref{alg:A1} converges to the global solution after $2$ iterations.
If we choose $u(x) := x_1^2 + x_2^2$ and $v(x) := 2x_2^2$ (Case 2) then it converges to the global solution after $4$ iterations. Note that, in the first case, $u$ and $v$ are only convex, while $u$ is strongly convex with the parameter $\rho^u=2$ and $v$ is convex in the second case.
The convergence behaviour is illustrated in Figure \ref{fig:F1}. Here, the left figure corresponds to Case 1 and the right one corresponds to Case 2.
\begin{figure}[ht]
\begin{center}
\includegraphics[width=0.98\textwidth,height=4.3cm]{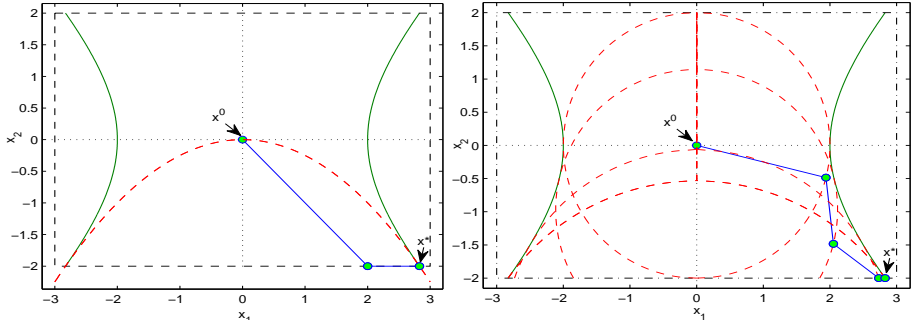}
\end{center}
\caption{Convergence behaviour of Algorithm \ref{alg:A1} using different DC decompositions.}\label{fig:F1}
\end{figure}

The following lemma shows that if Algorithm \ref{alg:A1} terminates after some iterations then $x^k$ is a stationary point of \eqref{eq:nlp_prob}.

\begin{lemma}\label{le:kkt_point}
Suppose that $x^k$ is a solution of \ref{eq:convex_subprob}; then it is a stationary point of the original problem \eqref{eq:nlp_prob}.
\end{lemma}

\begin{proof}
Suppose that $x^k$ is a solution of \ref{eq:convex_subprob} corresponding to the multiplier $\lambda^k$ then $(x^k,\lambda^k)$ is a solution of its KKT system, i.e., $x^k\in\Omega$, $0\in \partial f(x^k) + [\partial u(x^k) - \Xi^k]^T\lambda^k + N_{\Omega}(x^k)$, $u(x^k) - v(x^k) - (\Xi^k)(x^k-x^k) \leq 0$, $\lambda^k\geq 0$ and $(u(x^k) - v(x^k) - \Xi^k)(x^k-x^k))^T\lambda^k=0$, which implies that $x^k\in\Omega$, $0\in \partial f(x^k) + [\partial u(x^k) - \partial v(x^k)]^T\lambda^k + N_{\Omega}(x^k)$, $u(x^k) - v(x^k)\leq 0$, $\lambda^k\geq 0$ and $[u(x^k) - v(x^k)]^T\lambda^k = 0$. The last five relations mean that $(x^k,\lambda^k)$ satisfies \eqref{eq:kkt_cond}. Thus $x^k$ is a stationary point of \eqref{eq:nlp_prob} corresponding to the multiplier $\lambda^k$. 
\eoproof
\end{proof}

\section {Global convergence of  the SCP algorithm with DC constraints}\label{sec:convergence} 
The next lemma gives us a key property to prove the global convergence of Algorithm \ref{alg:A1}.
\begin{lemma}\label{le:descent_dir}
Suppose that $f$, $u_i$ and $v_i$ $(i=1,\dots,m)$ are $\rho^f$, $\rho^{u_i}$ and $\rho^{v_i}$ - convex, respectively. Then the sequence $\{(x^k,\lambda^{k})\}$ generated by Algorithm \ref{alg:A1} satisfies
\begin{eqnarray}
f(x^k) - f(x^{k+1}) &&\geq \frac{1}{2}(\rho^f + \sum_{i=1}^m\rho^{u_i}\lambda^{k+1}_i)\norm{x^{k+1} - x^k}^2 \nonumber\\
[-1.5ex]\label{eq:descent_dir}\\[-1.5ex]
&& + \frac{1}{2}\sum_{i=1}^m\rho^{v_i}\lambda^{k+1}_i\norm{x^k-x^{k-1}}^2.\nonumber
\end{eqnarray}
\end{lemma}

\begin{proof}
Since $x^{k+1}$ is a solution of \ref{eq:convex_subprob} corresponding to the multiplier $\lambda^{k+1}$, the KKT condition of \ref{eq:convex_subprob} is expressed as follows:
\begin{equation}\label{eq:kkt_subprob}
\begin{cases}
0 \in \partial f(x^{k+1})  + [\partial u(x^{k+1}) - \Xi^k]^T\lambda^{k+1} + N_{\Omega}(x^{k+1}),\\
0 \geq u(x^{k+1}) - v(x^k) - \Xi^k(x^{k+1}-x^k),  ~ \lambda^{k+1}\geq 0, \\
0 = (\lambda^{k+1})^T[u(x^{k+1}) - v(x^k) - \Xi^k(x^{k+1}-x^k)].
\end{cases}
\end{equation}
From the first line of \eqref{eq:kkt_subprob}, we have
\begin{equation}\label{eq:term_21}
(\xi_f^{k+1})^T(y-x^{k+1}) + (\lambda^{k+1})^T[\Xi_u^{k+1} - \Xi^k](y - x^{k+1}) \geq 0, ~\forall y\in \Omega,
\end{equation}
for all vectors $\xi_f^{k+1} \in \partial f(x^{k+1})$ and matrices $\Xi_u^{k+1} \in \partial u(x^{k+1})$.

Since $f$ and $u_i$ $(i=1,\dots,m)$ are strongly convex on $\Omega$, it holds that
\begin{align}
&f(y)  - f(x^{k+1}) \geq (\xi_f^{k+1})^T( y - x^{k+1}) + \frac{\rho^f}{2}\norm{y - x^{k+1}}^2,~ \forall y\in\Omega, \label{eq:term_22}\\
&u(y) - u(x^{k+1}) \geq \Xi_u^{k+1}( y - x^{k+1}) + \frac{\rho^u}{2}\norm{y - x^{k+1}}^2,~ \forall y\in\Omega, \label{eq:term_23}
\end{align}
where $\rho^u = (\rho^{u_1},\dots, \rho^{u_m})^T$.
Combining \eqref{eq:term_21}, \eqref{eq:term_22} and \eqref{eq:term_23}, and noting that $\lambda^{k+1}\geq 0$, we obtain
\begin{eqnarray}
&&f(y) - f(x^{k+1}) + (\lambda^{k+1})^T[u(y) - u(x^{k+1}) - \Xi^k(y-x^{k+1})]  \nonumber\\
&&\geq (\xi_f^{k+1})^T(y-x^{k+1}) + (\lambda^{k+1})^T(\Xi_u^{k+1} - \Xi^k)(y - x^{k+1}) \nonumber\\
[-1.5ex]\label{eq:term_24}\\[-1.5ex]
&&+ \frac{1}{2}[\rho^f + \sum_{i=1}^m\rho^{u_i}\lambda^{k+1}_i]\norm{y-x^{k+1}}^2 \nonumber\\
&&\geq \frac{1}{2}[\rho^f + \sum_{i=1}^m\rho^{u_i}\lambda^{k+1}_i]\norm{y-x^{k+1}}^2, ~\forall y\in\Omega.\nonumber
\end{eqnarray}
Substituting $y=x^k\in\Omega$ into \eqref{eq:term_24} and after a simple rearrangement, we get
\begin{eqnarray}
&& f(x^k) + (\lambda^{k+1})^T[u(x^k) - v(x^k)] \nonumber\\
&&-  f(x^{k+1}) - (\lambda^{k+1})^T[u(x^{k+1}) - v(x^k) - \Xi^k(x^{k+1}-x^k)] \label{eq:term_24a}\\
&& \geq \frac{1}{2}[\rho^f + \sum_{i=1}^m\rho^{u_i}\lambda^{k+1}_i]\norm{x^{k+1}-x^k}^2.\nonumber
\end{eqnarray}
Using the third line of \eqref{eq:kkt_subprob}, the inequality \eqref{eq:term_24a} is reduced to
\begin{equation}\label{eq:term_24b}
f(x^k) + (\lambda^{k+1})^T[u(x^k) - v(x^k)] -  f(x^{k+1}) \geq \frac{1}{2}[\rho^f + \sum_{i=1}^m\rho^{u_i}\lambda^{k+1}_i]\norm{x^{k+1}-x^k}^2.
\end{equation}
Now, since $v_i$ $(i=1,\dots,m)$ is $\rho^{v_i}$ - convex, we have
\begin{equation*}\label{eq:term_25}
v(x^{k+1}) - v(x^k) \geq \Xi^k(x^{k+1}-x^k) + \frac{\rho^v}{2}\norm{x^{k+1}-x^k}^2,
\end{equation*}
where $\rho^v = (\rho^{v_1},\dots,\rho^{v_m})^T$. 
This inequality implies that
\begin{equation}\label{eq:term_26}
u(x^{k+1}) - v(x^{k+1}) \leq  u(x^{k+1}) - v(x^k) - \Xi^k(x^{k+1}-x^k) - \frac{\rho^v}{2}\norm{x^{k+1}-x^k}^2.
\end{equation}
Using the second line of \eqref{eq:kkt_subprob} for \eqref{eq:term_26}, we obtain
\begin{equation}\label{eq:term_27}
u(x^{k+1}) - v(x^{k+1}) \leq - \frac{\rho^v}{2}\norm{x^{k+1}-x^k}^2 \leq 0.
\end{equation}
Applying \eqref{eq:term_27} with $x^{k}$ instead of $x^{k+1}$ to \eqref{eq:term_24b} yields
\begin{equation}\label{eq:term_27b}
f(x^k)  -  f(x^{k+1}) \geq \frac{1}{2}[\rho^f + \sum_{i=1}^m\rho^{u_i}\lambda^{k+1}_i]\norm{x^{k+1}-x^k}^2 + \frac{1}{2}\sum_{i=1}^m\rho^{v_i}\lambda^{k+1}_i\norm{x^k-x^{k-1}}^2,
\end{equation}
which proves \eqref{eq:descent_dir}.
\eoproof
\end{proof}

\begin{remark}\label{re:feasible_point}
From the proof of Lemma \ref{le:descent_dir} (see \eqref{eq:term_27}) we can see that Algorithm \ref{alg:A1} always generates a feasible sequence $\{x^k\}$ to \eqref{eq:nlp_prob}. If $\rho^v > 0$ then it is strictly feasible.  
Thus Algorithm \ref{alg:A1} can be considered as an {\it inner approximation method}. 
\end{remark}

\begin{remark}\label{re:strong_convex}
If either $f$ is strongly convex or at least one function $u_i$ (reps., $v_i$) $(i=1,\dots, m)$ with respect to $\lambda^{k+1}_i>0$ is strongly convex then the sequence of the objective values $\{f(x^k)\}$ is decreasing. 
\end{remark}

The convergence of Algorithm \ref{alg:A1} is stated by the following result.

\begin{theorem}\label{th:convegence_theorem}
Suppose that $f$ is bounded from below on $D$, 
and the sequence $\{(x^k,\lambda^k)\}$ generated by Algorithm \ref{alg:A1} is bounded on $\Omega\times\mathbf{R}^m_{+}$. Then:
\begin{itemize}
\item[(i)] If  $\rho^f>0$ then $\lim_{k\to\infty}\norm{x^{k+1}-x^k} = 0$, and every accumulation point $(x^{*},\lambda^{*})$ of $\{(x^k,\lambda^k)\}$ is a KKT point of \eqref{eq:nlp_prob}.
\item[(ii)] If there exists an index $i_0\in \{1, \dots, m\}$ such that $\rho^{u_{i_0}} > 0$ (resp., $\rho^{v_{i_0}} > 0$) then
\begin{equation*}
\lim_{k\to\infty}\lambda^{k+1}_{i_0}\norm{x^{k+1}-x^k}^2 = 0 ~~ (\textrm{resp.,}~ \lim_{k\to\infty}\lambda^{k+1}_{i_0}\norm{x^{k}-x^{k-1}}^2 = 0 ),  
\end{equation*}
and every accumulation point of $(x^{*},\lambda^{*})$ of $\{(x^k,\lambda^k)\}$ such that $\lambda^{*}_{i_0} > 0$ is a KKT point of \eqref{eq:nlp_prob}.
\item[{(iii)}] If the set of the KKT points of \eqref{eq:nlp_prob} is finite then the whole sequence $\{ (x^k,\lambda^k) \}$ converges to a KKT point of \eqref{eq:nlp_prob}.
\end{itemize}
\end{theorem}

\begin{proof}
From Lemma \ref{le:descent_dir}, it turns out  that the sequence $\{f(x^k)\}$ is nonincreasing and is bounded from below by assumption. Then it converges to $f^{*} >-\infty$. Summing up inequality \eqref{eq:descent_dir} from $k=1$ to $k=N$ and then passing to the limit as $k\to\infty$ we obtain
\begin{align}\label{eq:thm1_tmp}
\sum_{k=1}^{\infty}\left[\frac{1}{2}(\rho^f \!+\! \sum_{i=1}^m\rho^{u_i}\lambda^{k+1}_i)\norm{x^{k+1} \!-\! x^k}^2
\! + \! \frac{1}{2}\sum_{i=1}^m\rho^{v_i}\lambda^{k+1}_i\norm{x^k \!-\! x^{k-1}}^2\right] \leq f(x^0) \!-\! f^{*} < +\infty. 
\end{align}
If $\rho^f > 0$ then the inequality \eqref{eq:thm1_tmp} implies that $\lim_{k\to\infty}\norm{x^{k+1}-x^k} = 0$. 
Since $\{(x^k, \lambda^k)\}$ is bounded by assumption, it has at least one limit point. Suppose that $(x^{*},\lambda^{*})$ is a limit point of $\{(x^k,\lambda^k)\}$, which means that there exists a subsequence $\{(x^k,\lambda^k)\}_{k\in\mathcal{K}}$ of $\{(x^k,\lambda^k)\}$ such that $(x^{k},\lambda^k)(k\in\mathcal{K})\to (x^{*},\lambda^{*})$, where $\lambda^{*}\in\mathbf{R}^m_{+}$. Since $\partial f$, $\partial u_i$ and $\partial v_i$ ($i=1,\dots,m$) are upper semicontinuous, passing to the limit of the subsequence as $k (\in\mathcal{K})\to\infty$ in \eqref{eq:kkt_subprob} we conclude that $(x^{*},\lambda^{*})$ is a KKT point of \eqref{eq:nlp_prob}. The statement (i) is proven.

For the statement (ii), it is sufficient to prove the first case (i.e., there exists $i_0$ such that $\rho^{u_{i_0}}>0$), the second case is done similarly. Suppose that there exists at least one index $i_0\in \{1, \dots, m\}$ such that $\rho^{u_{i_0}} > 0$. 
Using again \eqref{eq:thm1_tmp}, it is easy to show that $\lim_{k\to\infty}\lambda^{k+1}_{i_0}\norm{x^{k+1}-x^k}^2 = 0$. As before, if $(x^{*},\lambda^{*})$ is a limit point of a subsequence $\{(x^k,\lambda^k)\}_{k\in\mathcal{K}}$ such that $\lambda^{*}_{i_0}>0$ then we have $\lim_{k(\in\mathcal{K})\to\infty}\norm{x^{k+1}-x^k} = 0$. Passing to the limit through the subsequence as $k(\in\mathcal{K})\to\infty$ in \eqref{eq:kkt_subprob} we conclude again that $(x^{*},\lambda^{*})$ is a KKT point of \eqref{eq:nlp_prob}. 

The last statement (iii) can be proved similarly using the same technique as in \cite{Ostrowski1966}[Chapt. 28].
\eoproof
\end{proof}

Suppose that $x^{*}$ is a stationary point of \eqref{eq:nlp_prob} associated with a multiplier $\lambda^{*}$. If we denote by
\begin{equation}\label{eq:inactive_index}
I_{+}(x^{*}):=\{ i\in\{1,\dots,m\}~|~ \lambda_i^{*} > 0\} 
\end{equation}
the strictly active set of \eqref{eq:nlp_prob} at $x^{*}$, then the assumption (ii) in Theorem \ref{th:convegence_theorem} requires that $I_{+}(x^{*})\neq\emptyset$ and at least one function $u_i$ (or $v_i$) $i\in I_{+}(x^{*})$ is strongly convex. 

\begin{remark}\label{re:regularization}({\it Regularization}). 
From Lemma \ref{le:descent_dir}, we see that if $f$, $u_i$ and $v_i$ are only convex for all $i\in I_{+}(x^{*})$ (but not strongly convex) then Algorithm \ref{alg:A1} might not make $f$ strictly decreasing, i.e, $f(x^{k+1}) \not< f(x^k)$ for $k\geq 0$. 
In order to overcome this issue, a regularization term can be added to the objective function of \ref{eq:convex_subprob}. Instead of solving  problem \ref{eq:convex_subprob}, Algorithm \ref{alg:A1} is modified at Step 2 by solving the following regularized problem:
\begin{equation}\label{eq:convex_subprob_regu}
\left\{\begin{array}{cl}
\displaystyle\min_{x\in\mathbf{R}^n} &f(x) + \frac{\rho}{2}\norm{x-x^k}^2\\
\textrm{s.t.} &u(x) - v(x^k) - \Xi^k(x-x^k) \leq 0, \tag{$\textrm{P}_{\textrm{r}}(x^k)$}\\
& x\in\Omega,
\end{array}\right.
\end{equation}
where $\Xi^k \in \partial v(x^k)$ is arbitrary, and $\rho>0$ is a regularization parameter. This technique is closely related to the \textit{proximal point methods} \cite{Lewis2008,Rockafellar1997}.   
\end{remark}
However, using the regularization term with a large $\rho$ may lead to short steps. Consequently, Algorithm \ref{alg:A1} converges slowly to a KKT point. 
In practice, we only add the regularization term if the solution of \ref{eq:convex_subprob} does not make $f$ strictly decreasing at the current iteration.
Note that if $\rho>0$ then Algorithm \ref{alg:A1} always makes $f$ strictly decreasing, i.e., $f(x^k) - f(x^{k+1}) \leq -\frac{\rho}{2}\norm{x^{k+1}-x^k}^2 < 0$ for $x^{k+1}\neq x^k$.  

\begin{remark}\label{re:DC_objective}({\it Handling the DC objective function}). 
If the objective function $f$ of \eqref{eq:nlp_prob} is also a DC function and $f(x) = f_1(x) - f_2(x)$ is a DC decomposition of $f$, then  subproblem \ref{eq:convex_subprob} at Step 2 of  Algorithm \ref{alg:A1} is replaced by the following convex subproblem:
\begin{equation}\label{eq:convex_subprob_dc_obj}
\left\{\begin{array}{cl}
\displaystyle\min_{x\in\mathbf{R}^n} &f_1(x) - f_2(x^k) - (\xi^k_{f_2})^T(x-x^k) \\
\textrm{s.t.} &u(x) - v(x^k) - \Xi^k(x-x^k) \leq 0, \tag{$\textrm{P}_{\textrm{dc}}(x^k)$}\\
& x\in\Omega,
\end{array}\right.
\end{equation}
with matrix $\Xi^k \in \partial v(x^k)$ and vector $\xi^k_{f_2}\in\partial f_2(x^k)$. 
The conclusions of Theorem \ref{th:convegence_theorem} are still valid for this modification. A smooth variant of this algorithm was considered in \cite{Smola2005} applied to DC programs arising in support vector machines, without convergence theory, however. 
\end{remark} 

\section{A relaxed SCP algorithm with DC constraints}\label{sec:SCP_penalty} 
According to DCA approaches, to handle DC constraints, a penalty function is used to bring these constraints into the objective function \cite{An1999}. 
The obtained problem becomes an unconstrained or convex constrained DC program, and the unconstrained DCA can be applied to solve this problem.
We start this section by introducing one possible DC decomposition to handle the DC constraints using $L_1$-penalty functions, which is often used in practice \cite{An2005,An1999}.
We will show through an example that by using an $L_1$-penalty function to handle the DC constraints, DCA may make only slow progress to a stationary point of \eqref{eq:nlp_prob}.  
  
Let us define the $L_1$-penalty function of \eqref{eq:nlp_prob} as follows:
\begin{equation}\label{eq:l1_penalty}
\phi(x; \mu) := f(x) + \mu \norm{[g(x)]_{+}}_1,
\end{equation}
where $\mu>0$ is a penalty parameter and $[g(x)]_{+} = \max\{ g(x), 0 \}$. 
Note that if $g$ has a DC decomposition $(u, v)$ then we have  $[g(x)]_{+} = \max\{u(x), v(x)\} - v(x)$. Since $u$ and $v$ are convex, $\max\{u, v\}$ is also convex. Thus $(\max\{u(\cdot), v(\cdot)\}, v(\cdot))$ is a DC decomposition of $[g(\cdot)]_{+}$. 
Since $\phi(x;\mu) = f(x) + \mu\sum_{i=1}^m[\max\{u_i(x), v_i(x)\} - v_i(x)] = f(x) + \mu\sum_{i=1}^m[\max\{u_i(x), v_i(x)\} - \mu v_i(x)]$, if we define $u_{\mu}(x) := f(x) + \mu\sum_{i=1}^m\max\{u_i(x), v_i(x)\}$ and $v_{\mu}(x) := \mu\sum_{i=1}^m v_i(x)$ then $\phi(x;\mu)$ is a DC function, and $(u_{\mu}, v_{\mu})$ is a DC decomposition of $\phi(x;\mu)$.
 
The $L_1$-penalized problem associated with \eqref{eq:nlp_prob} can be rewritten as a convex constrained DC program:
\begin{equation}\label{eq:penalty_prob}
\min_{x\in \Omega} \big\{\phi(x;\mu) = u_{\mu}(x) - v_{\mu}(x)\big\}. \tag{$\textrm{P}^{\textrm{uc}}_{\mu}$}  
\end{equation}
DCA \cite{An1999} starts from an initial point $x^0\in\Omega$ and generates a sequence $\{ x^k \}$ by solving the  following convex subproblem:
\begin{equation}\label{eq:penalty_subprob}
\min_{x\in \Omega} u_{\mu}(x) - v_{\mu}(x^k) - (\xi^k_{\mu})^T(x-x^k), \tag{$\textrm{P}^{\textrm{uc}}_{\mu}(x^k)$} 
\end{equation}
where $\xi_{\mu}^k \in \partial v_{\mu}(x^{k})$ and $\mu$ is fixed to a suitable large value.
It is proved in \cite{An1999} that for this DC decomposition, there exists an exact penalty parameter $\mu_l>0$ such that for all $\mu \geq \mu_l$, any solution of problem \eqref{eq:penalty_prob} solves \eqref{eq:nlp_prob}. 

Now, we show that by using this standard technique, DCA may lead to slow convergence to a stationary point. 
Indeed, we consider an example by minimizing a convex function $f$ subject to a DC quadratic constraint $\frac{1}{2}(x^TPx - x^TQx) \leq p^Tx + r$, where matrix $P$ is symmetric positive semidefinite, $Q$ is symmetric positive definite, $p\in\mathbf{R}^n$, and $r\in\mathbf{R}$. If we define $u(x) := \frac{1}{2}x^TPx - p^Tx - r$ and $v(x) := \frac{1}{2}x^TQx$ then $v$ is strongly convex with parameter $\rho^v = \lambda_{\min}(Q)$, where $\lambda_{\min}(Q)$ is the smallest eigenvalue of $Q$. 
Applying DCA to problem \eqref{eq:penalty_prob} we have $v_{\mu}(x) = \mu v(x)$ that is strongly convex with parameter $\rho^{v_{\mu}} = \mu \lambda_{\min}(Q)$. If $\mu$ is large then $\rho^{v_{\mu}}$ is also large. In this case, DCA makes only slow progress to a stationary point of \eqref{eq:nlp_prob}. 

Instead of using the penalty function \eqref{eq:l1_penalty} directly, in the SCP framework, we automatically obtain a different relaxed algorithm. 
We first relax the DC constraints by  
\begin{eqnarray}
&&\min_{x,s}f(x) + \mu\sum_{i=1}^ms_i\nonumber\\
&&\text{s.t.}~~ u(x) - v(x) \leq s, \label{eq:relaxed_prob}\\
&&{~~~~~~~}x\in\Omega, ~s\geq 0. 
\end{eqnarray}
We use a relaxation technique to handle possibly inconsistent linearizations that may lead to infeasibility of the convex subproblem \ref{eq:convex_subprob} in Algorithm \ref{alg:A1}. 
Note that $u(x)-s$ is convex in $(x,s)$ as well as $v(x)$. Each SCP-DC subproblem is then given by:
\begin{equation}\label{eq:convex_subprob_mu}
\left\{\begin{array}{cl}
\displaystyle\min_{x\in\mathbf{R}^n} &f(x) + \mu \sum_{i=1}^m s_i\\
\textrm{s.t.} &u(x) - v(x^k) - \Xi^k(x-x^k) \leq s, \tag{$\textrm{P}(x^k;\mu)$}\\
& s \geq 0, ~ x\in\Omega.
\end{array}\right.
\end{equation}
A relaxed variant of Algorithm \ref{alg:A1} called \textit{relaxed SCP algorithm with DC constraints} (rSCP-DC) is described as follows:
  
\noindent\rule[1pt]{\textwidth}{1.0pt}{~~}
\begin{algorithm}\vskip -0.3cm\label{alg:A2}{~}
\end{algorithm}
\vskip -0.4cm
\noindent\rule[1pt]{\textwidth}{0.5pt}

\noindent{\bf Initialization:} Choose a penalty parameter $\mu_0 > 0$. 
                               Take an initial point $x^0$ in $\Omega$. Set $k:=0$.\\
\noindent{\bf Iteration $k$:}  For a given $x^k$, execute the three steps below:
\begin{itemize}
\item[]\textit{Step 1}: Compute a subgradient matrix $\Xi^k \in \partial v(x^k)$.
\item[]\textit{Step 2}: Solve the convex subproblem \ref{eq:convex_subprob_mu} with $\mu = \mu_k$
to get a solution $(x^{k+1}, s^{k+1})$ and the corresponding multiplier $\lambda^{k+1}$.
\item[]\textit{Step 3}: If $\norm{x^{k+1} - x^k} \leq \varepsilon$ and $\norm{s^{k+1}}\leq\varepsilon$ with a given tolerance $\varepsilon>0$, then stop. Otherwise, update the parameter $\mu_k$, increase $k$ by $1$ and go back to Step 1.
\end{itemize}
\vskip-0.3cm
\noindent\rule[1pt]{\textwidth}{1.0pt}
Note that the subproblem \ref{eq:convex_subprob_mu} is always feasible and the convergence theory of the previous section is applicable. 
However, the parameter $\mu$ influences the behaviour of Algorithm \ref{alg:A2}.  
If the parameter $\mu$ is chosen too large, the minimization enforces $s$ to decrease, which reduces the infeasibility gap of the subproblems \ref{eq:convex_subprob_mu} too fast. Otherwise, the infeasibility gap $s$ may be increased. Balancing between the optimality and the infeasibility plays an important role in Algorithm \ref{alg:A2}. 
The parameter $\mu_k$ can be fixed to a ``suitable'' value or updated at each iteration of the algorithm. A refined variant, which is however not the topic of this paper, separately updates penalty parameters $\mu_i$ for each $s_i$ and make sure that they are sufficiently large, but not much larger than the corresponding constraint multipliers.  

The following inequality shows that Algorithm \ref{alg:A2} makes a decreasing progress of the objective function $f_{\mu}(x, s) := f(x) + \mu\sum_{i=1}^ms_i$.

\begin{corollary}\label{co:descent_lemma_for_penalty}
Suppose that $f$, $u_i$ and $v_i$ $(i=1,\dots,m)$ are $\rho^f$, $\rho_i^{u}$ and $\rho_i^{v}$ - convex, respectively. Then the sequence $\{ (x^k,\lambda^{k}, s^k) \}$ generated by Algorithm \ref{alg:A2} satisfies
\begin{eqnarray}
f_{\mu}(x^k, s^k) - f_{\mu}(x^{k+1}, s^{k+1}) &&\geq \frac{1}{2}(\rho^f + \sum_{i=1}^m\rho^{u_i}\lambda^{k+1}_i)\norm{x^{k+1} - x^k}^2 \nonumber\\
[-1.5ex]\label{eq:descent_dir_for_penalty}\\[-1.5ex]
&& + \frac{1}{2}\sum_{i=1}^m\rho^{v_i}\lambda^{k+1}_i\norm{x^k-x^{k-1}}^2,\nonumber
\end{eqnarray}
where $f_{\mu}(x, s) := f(x) + \mu\sum_{i=1}^ms_i$.
\end{corollary}

The conclusions of Theorem \ref{th:convegence_theorem} still hold for this case, where the objective function is $f_{\mu}(x, s)$ (with a fixed value $\mu>0$) instead of $f$.

= 
\section{Numerical tests}\label{sec:numerical_results} 
To verify the performance of Algorithms \ref{alg:A1} and \ref{alg:A2}, we implement two numerical examples. The first example solves nonconvex quadratically constrained quadratic programs (ncvQCQP). The second one is a mathematical program with complementarity constraints.
  
\subsection{Example 1}\label{exam:QCQP}
Consider the following indefinite quadratically constrained quadratic programming problem:
\begin{equation}\label{eq:ncvQCQP}
\left\{\begin{array}{cl}
\displaystyle\min_{x\in\mathbf{R}^n} &f(x):=\frac{1}{2}x^TQx + q^Tx \\
\textrm{s.t.} &\frac{1}{2}x^TPx + p^Tx \leq \alpha,\\
              &Ax \leq b,\\
              &\underline{l}\leq x\leq \bar{u},
\end{array}\right.\tag{$\textrm{ncvQCQP}$}
\end{equation}
where $q, p, \underline{l}, \bar{u} \in\mathbf{R}^n$, $\alpha\in\mathbf{R}$, $b\in\mathbf{R}^{m_2}$, $Q$ is a symmetric positive semidefinite matrix in $\mathbf{R}^{n\times n}$, $A$ is an $m_2\times n$ real matrix, and $P$ is an $n\times n$ symmetric indefinite matrix.
If $P$ is symmetric positive semidefinite then problem \eqref{eq:ncvQCQP} is a convex quadratically constrained quadratic programming problem (QCQP) \cite{Boyd2004}.    
 
We first test Algorithms 1 and 2 with some random data in $[-10, 10]$ and compare the performance with the built-in Matlab solver \textsc{fmincon} for $10$ problems. 
The data is created as follows:
\begin{itemize}
\item Generate a random matrix $M$ and compute $Q := M^TM + 0.5I_n$, where $I_n$ is the identity matrix in $\mathbf{R}^{n\times n}$.
\item Vectors $q$, $p$, $b$ and matrix $A$ are random in $[-10,10]$, and $\alpha = 10$.
\item Generate a random matrix $P_{\textrm{r}}$ in $[-10, 10]$ and then compute $P:= 0.5(P_{\textrm{r}} + P_{\textrm{r}}^T)$. 
\item The lower bound vector $\underline{l}$ and the upper bound vector $\bar{u}$  are given by $(-5,\dots, -5)^T$ and $(10,\dots, 10)^T$, respectively.
\end{itemize}
Since every symmetric matrix $P$ can be decomposed as $P = P_1-P_2$, where $P_1$ and $P_2$ are symmetric positive semidefinite (using spectral decompositions).
The constraint
\begin{equation*}
\frac{1}{2}x^TPx + p^Tx \leq \alpha 
\end{equation*}
is expressed as a DC constraint:
\begin{equation*}
\frac{1}{2}x^TP_1x + p^Tx  - \frac{1}{2}x^TP_2x \leq \alpha, 
\end{equation*}
where $P_1 = V\Sigma_{+}V^T$ and $P_2=V\Sigma_{-}V^T$ with $\Sigma_{+} = \text{diag}( \sigma^{+}_i )$ and $\Sigma_{-} = \text{diag}( \sigma^{-}_i )$, $\sigma^{+}_i = \max\{\sigma_i, 0\}$, $\sigma^{-}_i = \max\{-\sigma_i, 0 \}$, $\sigma_i$ is the $i^{\text{th}}$ eigenvalue of matrix $P$, and $V$ is a matrix whose columns are formed by the eigenvectors of  $P$.

We implement Algorithms \ref{alg:A1} and \ref{alg:A2} in Matlab 7.8.0 (R2009a) running on a PC desktop with Intel(R) Core(TM)2 Quad CPU Q6600 2.4GHz, 3Gb RAM. We use the same DC decomposition of the DC constraint in both algorithms. 
To solve the convex quadratic subproblems, we use the CVX package (with Sedumi as a solver)\footnote{Available at: http://cvxr.com/cvx/}. The tolerance is given by $\varepsilon=10^{-6}$ and the penalty parameter $\mu_k$ is fixed to a certain value in Algorithm \ref{alg:A2} (see Tables \ref{tb:example_01} and \ref{tb:example_01.2}).
The computational results are reported in Table \ref{tb:example_01}.
\begin{table}
\begin{center}
\caption{Computational results of Algorithms \ref{alg:A1} and \ref{alg:A2} for \eqref{eq:ncvQCQP}.}\label{tb:example_01}
{\footnotesize
\begin{tabular}{|c|c|r|c|c|r|c|r|c|c|c|r|}\hline\hline
\multicolumn{4}{|c|}{Problem Information} & \multicolumn{2}{|c|}{Algorithm \ref{alg:A1}} & \multicolumn{4}{|c|}{Algorithm \ref{alg:A2}} & \multicolumn{2}{|c|}{\textsc{fmincon}}\\ \hline
$\texttt{m}_2$ & \texttt{n}  & $f^{*}$ & \texttt{error} &\!\!\texttt{iter}\!\!& \texttt{time} & \!\!\texttt{iter}\!\! & \texttt{time} & $\mu_k$ & $\rho$ & \!\!\texttt{iter}\!\! & \texttt{time}\\  \hline
 5 &  10 &  121.3768 & \!\!$3\!\times\! 10^{-4}$\!\!  &  64 &  16.13 &   3 &   1.88 & 0.1  & 0 &  36  &     1.54  \\ \hline
10 &  30 &   12.1228 & \!\!$3\!\times\! 10^{-4}$\!\!  &  63 &  18.18 &   3 &   1.37 & 0.1  & 0 &  68  &    11.89  \\ \hline
10 &  50 &   -1.5614 & \!\!$2\!\times\! 10^{-4}$\!\!  &  72 &  44.19 &   4 &   3.05 & 0.1  & 0 &  86  &    49.76  \\ \hline
10 & 100 &   -1.2812 & \!\!$3\!\times\! 10^{-4}$\!\!  &  67 &  58.38 &   4 &   3.51 & 0.1  & 0 & 178  &   664.41  \\ \hline
20 & 100 &   13.5225 & \!\!$2\!\times\! 10^{-4}$\!\!  &  72 &  56.68 &   3 &   2.83 & 0.1  & 0 & 217  &   747.50  \\ \hline
20 & 200 &   -2.3946 & \!\!$2\!\times\! 10^{-4}$\!\!  &  76 &  90.29 &   4 &   4.47 & 0.1  & 0 & 400  & 17328.78  \\\hline
30 & 200 &    3.3814 & \!\!$3\!\times\! 10^{-4}$\!\!  &  68 &  81.06 &   4 &   6.62 & 0.1  & 0 & 290  & 12937.88  \\\hline
30 & 300 &    7.0023 & \!\!$3\!\times\! 10^{-4}$\!\!  &  96 & 203.63 &   4 &  10.47 & 0.1  & 0 & $\#$ &  $\#$     \\\hline
40 & 400 &   17.2517 & \!\!$3\!\times\! 10^{-4}$\!\!  &  77 & 274.58 &   4 &  15.66 & 0.1  & 0 & $\#$ &  $\#$     \\\hline
50 & 500 &   44.5623 & \!\!$2\!\times\! 10^{-4}$\!\!  & 100 & 612.48 &   4 &  25.77 & 0.1  & 0 & $\#$ &  $\#$     \\\hline\hline
\end{tabular}}
\vskip -0.3cm
\end{center}
\end{table}
For comparison, we solve three problems taken from \cite{Floudas2007}.  The two first problems ($P_1$, $P_2$) are in Chapter 3\cite{Floudas2007}[test problems 1 and 2, respectively], while the last one ($P_3$) is a VLSI design problem in Chapter 3\cite{Floudas2007}[test problem 2].
The best-known solutions and optimal values of $P_1$, $P_2$ are given in \cite{Floudas2007}:
\begin{align*}
&x_1^{*} = (579.31, 1359.97, 5109.97, 182.02, 295.6, 217.98, 286.42, 395.60)^T, ~ f^{*}_1 = 7049.25,\\
&x^{*}_2 = (78, 33, 29.9953, 45, 36.7758)^T, ~f_2^{*}= 30665.5387,  
\end{align*}
respectively. 
The optimal value of $P_3$ is $f^{*}_3 = 146.25$.  
Our computational results for these problems are reported in Table \ref{tb:example_01.2}, which closely approximate to the best-known solution reported in \cite{Floudas2007}.

\begin{table}
\begin{center} 
\caption{Computational results of Algorithms \ref{alg:A1} and \ref{alg:A2} for three nonconvex QP problems in \cite{Floudas2007}.}\label{tb:example_01.2}
{\footnotesize
\begin{tabular}{|c|c|c|c|c|c|c|c|c|}\hline\hline
\multicolumn{4}{|c|}{Problem Information} & \multicolumn{5}{|c|}{Algorithm \ref{alg:A1} (or \ref{alg:A2})}\\ \hline
$\texttt{N}^0$ & {\scriptsize $[\texttt{n}, \texttt{m}_1, \texttt{m}_2, \texttt{m}_3]$} & \textrm{otype} & $f^{*}$ (in \cite{Floudas2007}) & \texttt{iter}  & \texttt{time} &$\mu_k$ & $\rho$ & $f^{*}$ \\ \hline   
$\texttt{P}_1$ &  [8,0,3,3] & \texttt{ln} & 7049.25 &  176 & 74.36 & $100$ & $1\times 10^{-3}$ & 7049.25352  \\ \hline
$\texttt{P}_2$ &  [5,0,0,6] & \texttt{nq} & -30665.5387 &   5  &  2.94 &  810 & $1\times 10^{-3}$ & -30665.53892  \\ \hline
$\texttt{P}_3$ &  [12,5,2,6] &\texttt{nq} & 146.25 &   5  &  2.76 &  0  &   0 & 146.25000  \\ \hline\hline
\end{tabular}}
\end{center}
\end{table}
The notations in Tables \ref{tb:example_01} and \ref{tb:example_01.2} include: $n$, $m_1$, $m_2$, $m_3$ are the size of the problems  (variables, linear equality, linear inequality and DC constraints, respectively), $f^{*}$ is the optimal value, \texttt{otype} is the type of the objective function (\texttt{ln} is linear, \texttt{nq} is nonconvex quadratic), \texttt{error} is the quantity $\norm{x^{k+1}-x^k}$, \texttt{iter} is the number of iterations, and \texttt{time} is the CPU time in seconds; $\mu$ and $\rho$ are the penalty and the regularization parameters in Algorithm \ref{alg:A2}, respectively.  The symbol $\#$ indicates that \textsc{fmincon} exceeds the limit time $T_{\max}= 4$ hours.

\subsection{Example 2}\label{exam:MPEC}
This example illustrates an application of Algorithm \ref{alg:A2} to solve mathematical programs with complementarity constraints presented in Section \ref{sec:examples}:
\begin{equation}\label{eq:aMPEC}
\left\{\begin{array}{cl}
\displaystyle\min_{x, y} & f(x, y)\\
\textrm{s.t.} &A x + By \leq a,\\
              &x^T(Cx + Dy + e ) = 0, ~ Cx + Dy + e \geq 0, ~x \geq 0,\\
              &x\in \Omega_x, ~ y\in\Omega_y,
\end{array}\right.\tag{MPCC}
\end{equation}
where $x\in\mathbf{R}^{n_x}$ is decision variable, $y\in\mathbf{R}^{n_y}$ is parameter, $f$ is convex with respect to $x$ and $y$, $A$, $B$, $C$ and $D$ are given consistent matrices, $a$ and $e$ are given consistent vectors, and $\Omega_x$, $\Omega_y$ are two convex sets in $\mathbf{R}^{n_x}$ and $\mathbf{R}^{n_y}$, respectively.

As in Example \ref{ex:motivating_exam2} of Section \ref{sec:examples}, we use a slack variable $z$ for $Cx + Dy + e \geq 0$, the complementarity condition of \eqref{eq:aMPEC} is expressed equivalently to
\begin{equation}\label{eq:CPcond}
x \geq 0, ~ z \geq 0, ~ x^Tz \leq 0, ~ Cx + Dy - z + e = 0. 
\end{equation}
Let us define a new variable $w = (x^T, y^T, z^T)^T \in \mathbf{R}^{n_w}$ with $n_w = 2n_x + n_y$, and denote by
$u(w) := \norm{(x,y,z)}^2$ and $v(w) := \norm{x-z}^2 + \norm{y}^2$, the third condition of \eqref{eq:CPcond} is equivalent to a DC constraint $u(w) - v(w) \leq 0$. Note that $u$ is strongly convex with parameter $\rho^u = 2$ and $v$ is only convex (not strongly convex).
We also define 
\begin{equation}\label{eq:omega_set}
\Omega_w := \left\{\begin{matrix} w = (x^T, y^T, z^T)^T \in \mathbf{R}^{n_w} ~|~ Ax + By \leq a, ~Cx + Dy - z + e = 0, \\
             x\in\Omega_x, ~y\in\Omega_y, ~x \geq 0, ~z\geq 0
\end{matrix}\right\}.  
\end{equation}
Since $\Omega_x$ and $\Omega_y$ are convex, and the remaining constraints are linear, $\Omega_w$ is convex in $\mathbf{R}^{n_w}$.

Problem \eqref{eq:aMPEC} is reformulated as
\begin{equation}\label{eq:nlp_prob_exam2}
\left\{\begin{array}{cl}
\displaystyle\min_{w\in\mathbf{R}^{n_w}} & f_w(w) := f(x, y)\\
\textrm{s.t.} &u(w) - v(w) \leq 0,\\
              &w \in \Omega_w,          
\end{array}\right.
\end{equation}
which coincides with \eqref{eq:nlp_prob}.

In this example, we implement Algorithm \ref{alg:A2} for solving three problems $\texttt{P}_7$, $\texttt{P}_9$ and $\texttt{P}_{10}$ in \cite{Facchinei1999}[problems 7, 9 and 10, respectively]. 
The parameter $\mu_k$ is fixed to $\mu = 10^{-1}$. To solve the convex subproblems \ref{eq:convex_subprob_mu} we also use the CVX package with  the Sedumi solver. 
For a given tolerance $\varepsilon = 10^{-6}$, the computational results are presented in Table \ref{tb:example_02}, which closely approximate to the results given in \cite{Facchinei1999}.
\begin{table}
\begin{center}
\caption{Computational results of Algorithm \ref{alg:A2} for \eqref{eq:aMPEC}.}\label{tb:example_02}
{\footnotesize
\begin{tabular}{|c|c|c|c|c|c|c|c|}\hline\hline
\multicolumn{1}{|c}{$\textrm{N}^o$} & \multicolumn{5}{|c|}{Problem Information} & \multicolumn{2}{|c|}{Algorithm \ref{alg:A2}} \\ \hline
& $[\texttt{m, n, l}]$ & $x^0$ & $f^{*}$ & \texttt{error} & \texttt{feasgap} & \texttt{iter} & \texttt{time} \\  \hline
$\texttt{P}_7$ &(2,2,6)& (40, 40) & $64.999$ & $7\times 10^{-7}$ & $5\times 10^{-11}$ & 9 & 9.91 \\ \hline
$\texttt{P}_9$ & (2,2,2) & (0, 0) & $7.095\times 10^{-12}$ & $1\times 10^{-5}$ & $4\times 10^{-15}$ & 18 & 13.00 \\ \hline
$\texttt{P}_9$ & - & (10, 0) & $1.351\times 10^{-11}$ & $2\times 10^{-5}$ & $1\times 10^{-10}$ & 18 & 12.82 \\ \hline
$\texttt{P}_9$ & - & (5, 5) & $1.294\times 10^{-11}$ & $2\times 10^{-5}$ & $1\times 10^{-10}$ & 18 & 12.83 \\ \hline
$\texttt{P}_9$ & - & (0, 10) & $1.229\times 10^{-11}$ & $2\times 10^{-5}$ & $1\times 10^{-11}$ & 18 & 12.91 \\ \hline
$\texttt{P}_9$ & - & (10, 10) & $2.597\times 10^{-11}$ & $8\times 10^{-6}$ & $1\times 10^{-11}$ & 19 & 13.56 \\ \hline
$\texttt{P}_{10}$ & (4,4,12) & $(5,5,15,15)$ & $-6600$ & $3\times 10^{-5}$ & $3\times 10^{-8}$ & 17 & 15.76 \\ \hline
\end{tabular}}
\end{center}
\end{table}
The solutions reported by Algorithm \ref{alg:A2} for $\texttt{P}_7$, $\texttt{P}_9$ and $\texttt{P}_{10}$ are
\begin{eqnarray*}
&&x^{*}_{\textrm{P}_7} = (25.00125, 30.00000)^T, ~~x^{*}_{\textrm{P}_9}=(10, 5)^T\\
&&\text{and}~ x^{*}_{\textrm{P}_{10}} = (7.515728, 3.77360, 11.48427, 17.22640)^T, 
\end{eqnarray*}
respectively.
Algorithm \ref{alg:A1} failed in this case because the set of interior points $\textrm{int}D$ of the feasible set $D$ is empty. 

\section{Conclusion}
The main aim of this paper is to investigate the relation between sequential convex programming (SCP) \cite{Lewis2008,Quoc2009b} and DC programming \cite{An2005,An1999,Pham1998}. We have provided a variant of the SCP algorithm for finding local minimizers of a nonconvex programming problem with DC constraints.
We have proved a global convergence theorem for this particular algorithm. Then we have addressed some extensions and proposed a relaxation technique to handle possibly inconsistent linearizations. 
Although finding a DC decomposition of a certain DC function is in general still a hard problem, in some applications (as we have shown in the examples) it is available or easy to compute. We have not concentrated on the local convergence. However, under mild assumptions, it had been proved in \cite{Quoc2009b} that the SCP method converges linearly to a KKT point of the original problem. Applications to nonconvex quadratic programming problems as well as mathematical programming problems with complementarity constraints have been presented through two numerical examples.

\vskip 0.2cm
{\small
\noindent{\textbf{Acknowledgments.}}

This research was supported by Research Council KUL: CoE EF/05/006 Optimization in Engineering(OPTEC), GOA AMBioRICS, IOF-SCORES4CHEM, several PhD/postdoc \& fellow grants; the Flemish Government via FWO: PhD/postdoc grants, projects G.0452.04, G.0499.04, G.0211.05, G.0226.06, G.0321.06, G.0302.07, G.0320.08 (convex MPC), G.0558.08 (Robust MHE), G.0557.08, G.0588.09, research communities (ICCoS, ANMMM, MLDM) and via IWT: PhD Grants, McKnow-E, Eureka-Flite+EU: ERNSI; FP7-HD-MPC (Collaborative Project STREP-grantnr. 223854), Contract Research: AMINAL, and Helmholtz Gemeinschaft: viCERP; Austria: ACCM, and the Belgian Federal Science Policy Office: IUAP P6/04 (DYSCO, Dynamical systems, control and optimization, 2007-2011).

}

\bibliographystyle{plain}

\end{document}